\documentclass[12pt,a4paper]{amsart}
\usepackage{booktabs}
\usepackage{comment}
\usepackage{amsmath}
\usepackage{amssymb}
\usepackage[disable]{todonotes}

\newlength\savedwidth
\def\toprule{\noalign{\global\savedwidth\arrayrulewidth
\global\arrayrulewidth 1pt}%
\hline
\noalign{\global\arrayrulewidth\savedwidth}}
\def\midrule{\noalign{\global\savedwidth\arrayrulewidth
\global\arrayrulewidth .6pt}%
\hline
\noalign{\global\arrayrulewidth\savedwidth}}
\def\bottomrule{\noalign{\global\savedwidth\arrayrulewidth
\global\arrayrulewidth 1pt}%
\hline
\noalign{\global\arrayrulewidth\savedwidth}}

\allowdisplaybreaks[1]

\newtheorem{theorem}{Theorem}
\newtheorem{proposition}{Proposition}

\newtheorem{lemma}{Lemma}

\theoremstyle{definition}

\theoremstyle{remark}

\newtheorem*{remark*}{Remark}

\DeclareMathOperator{\re}{Re}
\DeclareMathOperator{\im}{Im}

\newcommand{\testf}{{C^{\infty}_0(D)}}

\renewcommand{\theenumi}{\alph{enumi}}

\begin{document}
\title[Slit holomorphic flows and GFF]{Slit holomorphic stochastic flows and Gaussian free field}
\author[G. Ivanov]{Georgy Ivanov$^{\dag}$}
\author[N.-G. Kang]{Nam-Gyu Kang$^{\ddag}$}
\author[A. Vasil'ev]{Alexander Vasil'ev$^{\sharp}$}

\address{\newline {\it Georgy~Ivanov and Alexander Vasil'ev:}  \medskip \newline Department of Mathematics \newline University of Bergen \newline P.O.~Box 7803,
Bergen NO-5020 \newline Norway}
\email{georgy.ivanov@math.uib.no}
\email{alexander.vasiliev@math.uib.no}

\
\address{\newline {\it Nam-Gyu Kang:}  \medskip \newline Department of Mathematical Sciences  \newline Seoul National University  \newline San56-1 Shinrim-dong Kwanak-gu Seoul 151-747  \newline South Korea}
\email{nkang@snu.ac.kr}

\thanks{The authors$^{\dag,\sharp}$ have been  supported by EU FP7 IRSES program STREVCOMS, grant no. PIRSES-GA-2013-612669, by the grants of the Norwegian Research Council  \#213440/BG,  \#239033/F20\@. The author$^{\dag}$ has also been supported by Meltzerfondet. The author$^{\ddag}$ has also been supported by Samsung Science and Technology Foundation (SSTF-BA1401-01).}





\begin{abstract}
It was realized recently that the chordal, radial and dipolar SLEs are special cases of a general slit holomorphic stochastic flow. We characterize those slit holomorphic stochastic flows which generate level lines of the Gaussian free field. In particular, we describe the modifications of the Gaussian free field (GFF) corresponding to the chordal and dipolar SLE with drifts. Finally, we develop a version of conformal field theory based on the background charge and Dirichlet boundary condition modifications of GFF  and present martingale-observables for these types of SLEs.
\end{abstract}

\maketitle

\section{Introduction}

Close connections between the Schramm-L\"owner evolution (SLE) and the Gaussian free field (GFF) were first discovered in~\cite{SSexplorer,sheffieldcontour1}. Works in which the connections were further investigated include~\cite{dub09, izyurov,sheffieldcontour2,sheffield2010zipper}. Naturally, these connections between SLE and CFT led to a new interpretation of SLE in terms of conformal field theory (CFT), which was given in \cite{Kang2012a,MK,Kang2013,Kang2013k}.

In the case of the  chordal SLE(4), the relation with GFF is relatively easy to explain. Informally speaking, the random SLE(4) curve can be viewed as a `level line' (or a `discontinuity line') of the field
\[
 \tilde{\Phi}(z)=\Phi(z) + u(z),\quad z \in \mathbb{H}.
\]
Here $\Phi$ denotes the Gaussian free field in the upper half-plane $\mathbb{H}$ with the zero Dirichlet boundary conditions, and $u$ is a harmonic function given in this case by $u(z) = \sqrt{2}\, \arg z$.

A key ingredient in the proof of this result is the formula
\begin{equation}
\label{eq:hadamardIntro}
 dG_{\mathbb{H}}(w_t(z_1),w_t(z_2)) = -\frac12 d \langle u(w_t(z_1)), u(w_t(z_2)) \rangle. 
\end{equation}
where $\{w_t\}_{t\geq 0}$ is the stochastic flow of the chordal SLE(4), and $G_{\mathbb{H}}(z_1,z_2)$ is Green's function in $\mathbb{H}$. The formula \eqref{eq:hadamardIntro} is sometimes referred to as Hadamard's variational formula.

In~\cite{ivanov2014general}, chordal, radial and dipolar SLE were realized as particular cases of the \emph{slit holomorphic stochastic flows}. A normalized slit holomorphic stochastic flow is the solution to the Stratonovich stochastic differential equation
\begin{equation}
\label{eq:shsf}
 \begin{cases}
  d w_t(z) = -b(w_t(z)) \, dt + \sqrt{\kappa}\,\sigma(w_t(z))\circ dB_t, \\
  w_0(z) = z,  
 \end{cases}z\in\mathbb{H},
\end{equation}
where $b$ is a semicomplete vector field of the form
\begin{equation}
\label{eq:bformula}
b(z) = -\frac{2}{z} -b_{-1}  - b_0 z - b_1 z^2, \quad b_{-1},\, b_0, \,b_1\in \mathbb{R}, 
\end{equation}
and $\sigma$ is a complete vector field
 \begin{equation}
\label{eq:sigmaformula}
\sigma(z) = -1 - \sigma_0 z - \sigma_1 z^2,\quad \sigma_0, \, \sigma_1 \in \mathbb{R}.
\end{equation}

In this paper we address the following problem: traces of which slit holomorphic stochastic flows can be regarded as `level lines' of GFF modifications? The formula~\eqref{eq:hadamardIntro} plays an important role in answering this question.

In \cite{sheffieldcontour1} and \cite{sheffieldcontour2} it was essential that the function $u$ is bounded. We relax this condition, which allows us to consider chordal SLE curves with drift as level lines of a GFF modification with an unbounded $u$. In particular, we show that the field
\[
 \tilde{\Phi}(z) =\Phi(z) + \alpha \sqrt{\frac{2}{\kappa}} \, \im z + \sqrt{\frac{8}{\kappa}}\, \arg z
\]
corresponds to the chordal SLE($\kappa$) driven by $\sqrt{\kappa} B_t+\alpha\, t$.

It may happen that one modification of GFF corresponds to several different slit holomorphic stochastic flows. We verify that the families of curves generated by such flows are indeed closely related to each other.

\medskip

\noindent\textbf{Acknowledgments.} The authors would like to thank Alexey Tochin for the discussions on the subject of this paper.

\section{Definitions}
\subsection{Slit holomorphic stochastic flows}
Let  the vector fields $b$ and $\sigma$ be as in~\eqref{eq:bformula} and \eqref{eq:sigmaformula}, respectively,  and let $\{w_t\}_{t\geq 0}$ denote the flow of~\eqref{eq:shsf}.
The flow $\{w_t\}_{t\geq 0}$ is called the \emph{slit holomorphic stochastic flow} driven by $b$  and $\sigma$.

Some of the properties of slit holomorphic stochastic flows are listed below (see \cite{ivanov2014general}):

\begin{enumerate}
 \item For $t\geq 0$, consider the random set
\[
 H_t = \{z \in \mathbb{H} : w_s(z) \textrm{ is defined for }  s\in [0,t]\}. 
\]
Then, for each $t\geq 0$, $H_t$ is a simply connected domain, and $w_t$ is a conformal isomorphism $w_t:H_t \to \mathbb{H}$. This $H_t$ is called the \emph{evolution domain} of the flow $\{w_t\}_{t\geq 0}$ at the time $t$.

\item Let $K_t = \mathbb{H}\setminus H_t$ for $t\geq 0$ ($K_t$ is the \emph{hull} of the flow \eqref{eq:shsf} at time $t$). Then the law of $K_t$ is locally absolutely continuous with respect to the law of SLE$(\kappa)$ hulls (see \cite[Theorem~9]{ivanov2014general} for the precise statement). 

\smallskip 
\noindent The remaining properties are consequences of the absolute continuity of laws. 
\item The hull $K_t$ is almost surely generated by a curve.
\item Let $\gamma_t$ be the curve generating $K_t$. Then, with probability 1,
\begin{itemize}
\smallskip\item if $0 \leq \kappa \leq 4,$ $\gamma$ is a simple curve,
\smallskip\item if $4 < \kappa < 8,$ $\gamma$ has self-intersections,
\smallskip\item if $\kappa \geq 8$, $\gamma$ is space-filling.
\end{itemize}
\end{enumerate}

The chordal SLE can be obtained as  a slit holomorphic stochastic flow with 
\[
  b_{-1} = b_0 = b_1 = \sigma_0 = \sigma_1 = 0.   
\]
After a conjugation with an appropriate conformal map, radial and dipolar SLE can also be treated as special cases of \eqref{eq:shsf}:  

\begin{table}[ht]  \renewcommand{\arraystretch}{2.5}
\begin{center}
\begin{tabular}{c|c|c}
\toprule
\hspace{1em} SLE type \hspace{1em}  &  \hspace{2em}  $b$  \hspace{2em}  &  \hspace{2em}  $\sigma$  \hspace{2em}  \\ 
\midrule 
Chordal  & $-\dfrac2z$ & $-1$ \\ 
Radial & $-\dfrac2z - \dfrac z2$ & $- 1 -  \dfrac14{z^2}$ \\ 
\phantom{$\dfrac1{\frac1{1_1}}$} Dipolar \phantom{$\dfrac1{\frac1{1_1}}$} & $ -\dfrac2z +  \dfrac z2$ & $ - 1 + \dfrac14{z^2}$\\ 
\bottomrule 
\end{tabular} 
\end{center}
\end{table}


\subsection{Gaussian free field}

Let $D$ be a simply connected hyperbolic domain in  $\mathbb{C}$, and let $\testf$ denote the space of \emph{test functions} in $D$. Recall that a \emph{(Schwartz) distribution} on $D$ is a continuous linear functional on $\testf$ (see \cite{bookAdams} for details).
Let $A$ denote the area measure in $D$. For $\mathfrak{p}, \mathfrak{q}\in C^{\infty}_0(D)$, and let  $(\mathfrak{p}, \mathfrak{q})_{\mathcal{E}(D)}$ denote the inner product, 
\[
(\mathfrak{p}, \mathfrak{q})_{\mathcal{E}(D)}:= \int_{D\times D}2\, G_{D}(z_1, z_2)\, \mathfrak{p}(z_1)\, \mathfrak{q}(z_2)\, dA(z_1)\, dA(z_2),
\]
where $G_D$ denotes Green's function in $D$. 
The \emph{electrostatic potential energy} of $\mathfrak{p}$ in $D$ is 
\[
 E(\mathfrak{p}) = \|\mathfrak{p}\|_{\mathcal{E}(D)} = \sqrt{(\mathfrak{p}, \mathfrak{p})_{\mathcal{E}(D)}}.
\]

An instance of the \emph{Gaussian free field} (GFF) on $D$ is a random Schwartz distribution $\Phi_D$ on $D$, such that for a test function $\mathfrak{p}\in C^{\infty}_0(D)$, the image $(\Phi_D, \mathfrak{p})$ is a Gaussian random variable with the zero mean and variance $\|\mathfrak{p}\|^2_{\mathcal{E}(D)}$. 
The law of GFF is the only probability law on the space of distributions such that for a test function $\mathfrak{p}\in \testf$ the random  variable $(\Phi_D, \mathbb{\mathfrak{p}})$ is Gaussian with
$ \mathsf{E} (\Phi_D, \mathbb{\mathfrak{p}}) = 0,$ and  $ \mathsf{E} (\Phi_D, \mathbb{\mathfrak{p}})^2 = \|\mathfrak{p}\|_{\mathcal{E}(D)}^2$.

\subsubsection{Extension of GFF}
Let $D_0 $ be a subdomain on $D$, and let $\Phi_{D_0}  :C^{\infty}_0(D_0 ) \to L^2(\Omega, \mathbb{P})$ be an instance of GFF on $D_0 $ with zero boundary conditions. Then there exists a unique extension of $\Phi_{D_0}  $ to $C^{\infty}_0(D)$
\[
\tilde{\Phi}_{D_0}   :\testf \to L^2(\Omega, \mathbb{P}),
\]
 such that for every  $\mathfrak{p}\in \testf,$ the random variable $(\tilde{\Phi}_{D_0}  , \mathfrak{p})$ is zero mean Gaussian and
\[
 \mathsf{E} (\Phi_{D_0}  , \mathfrak{p})^2 = \|\mathfrak{p}\|^2_{\mathcal{E}(D_0 )}.
\]
Therefore, we can treat a Gaussian free field on a subdomain of $D$ as a distribution on the whole of $D$. \todo{careful: avoid confusion between laws and instances}

\subsubsection{Modifications of GFF} \todo{The term ``modification'' has here a different meaning, compared to works of Kang and Makarov} Let $u$ be a harmonic function in $D$. In particular, $u$ is continuous and $u\in L^1_{loc}(D,A)$. We can define the GFF \emph{with the mean $u$}  as the sum $\tilde{\Phi}_D = \Phi_D + u$, so that it acts on  test functions $\mathfrak{p}\in C^{\infty}_0(D)$ according to the rule
\[
 (\tilde{\Phi}_D, \mathfrak{p}) = (\Phi_D, \mathfrak{p}) + (u, \mathfrak{p}).
\]

If $u$ is bounded, then by Fatou's theorem it is the Poisson integral of a bounded measurable function $u_{\partial}$ on $\partial D$. In this case we can refer to $\tilde{\Phi}_D$ as the Gaussian free field with Dirichlet boundary conditions $u_{\partial}$. 

If $u$ is only bounded from above or from below, then the Herglotz representation theorem implies that it can be represented as the Poisson integral of a signed measure $\mu_u$ on $\partial D$.

Now, suppose $D_0 $ is a simply connected subdomain of $D,$ and the Lebesgue measure of $D\setminus D_0 $ is zero. Let $\tilde{\Phi}_{D_0}   = \Phi_{D_0}   + u$ be a modification of GFF on $D_0 $, and let $\mathfrak{p}\in C^{\infty}_0(D)$. Even though $u$ is only defined  on $D_0 $, the integral 
\[
 (u, \mathfrak{p}) = \int_{D} u(z)\, \mathfrak{p}(z) \, dA(z)
\]
is still well-defined, hence $\tilde{\Phi}_{D_0 }$ can be treated as a random distribution on the whole of $D$.

\subsubsection{Pullbacks of distributions and conformal invariance of {\rm GFF}} Let $\phi:D_1\to D_2$ be a conformal isomorphism, and let $\Psi$ be a distribution on $D_2$. The \emph{pullback}  $\Psi\circ \phi$ of $\Psi$ with respect to $\phi$ is the distribution on $D_1$ defined by the formula
\[
 (\Psi \circ \phi, \mathfrak{p}) = (\Psi, |(\phi^{-1})'|^2 \, \mathfrak{p}(\phi^{-1})), \quad \mathfrak{p}\in C^\infty_0(D_1).
\]
If $\Psi$ is an $L^1_{loc}(D_2)$ function, then this definition corresponds to the change of variables in the formula
\[
 \int_{D_1} \Psi(\phi (z))\, \mathfrak{p}(z) \,dA (z) = \int_{D_2} \Psi(w) \, |(\phi^{-1}(w))'|^2\, \mathfrak{p}(\phi^{-1}(w))\, dA(w).
\]
The Gaussian free field is \emph{conformally invariant}, meaning that $\Phi_{D_2} \circ \phi$ has the same law as an instance of GFF in $D_1$.
 
\subsection{Lie derivatives }
Let $v$ be a holomorphic vector field defined in a simply connected domain $D \subsetneq\mathbb{C}$ with a local flow $\psi_t.$ 
Suppose $f$ is a non-random conformal field, i.e., an assignment of $(f\,\|\,\phi):\phi(U)\to \mathbb{C}$ to each local chart $\phi:U\to\phi(U).$
The Lie derivative $\mathcal{L}_v f$ of $f$ is defined by 
$$\mathcal{L}_v f = \frac{d}{dt}\Big|_{t=0} f_t, \qquad 
(f_t\,\|\,\mathrm{id})(z) = (f\,\|\,\psi_{-t}),$$
where $\mathrm{id}$ is the identity chart of $D.$ 
If $f = f_D$ is a non-random smooth $(\lambda,\lambda_*)$-differential in $D,$ then
$$f_D = (h')^\lambda (\overline{h'})^{\lambda_*} f_{\widetilde D}\circ h,$$
where $h: D\to \widetilde D$ is a conformal transformation from $D$ onto another simply connected domain $\widetilde D \subsetneq\mathbb{C}.$ 
In this case, we have
$$f_t(z) = (\psi_t'(z))^\lambda  \overline{(\psi_t'(z))}^{\lambda_*} f(\psi_t(z)).$$
Thus Lie derivative $\mathcal{L}_v f$ of a $(\lambda,\lambda_*)$-differential $f$ is a differential:
$$\mathcal{L}_{v} f = \left(v \partial + \bar{v} \bar{\partial} + \lambda v' + \lambda_* \overline{v'}\right) f.$$
The formulas above can be generalized to $n$ complex variables, simply by extending the definition of $\mathcal{L}_{v}.$ For example, if $f = f(z_1,\ldots, z_n)$ is a $(\lambda_j,\lambda_{*j})$-differential at $z_j$ then
$$\mathcal{L}_{v} f (z_1, \ldots, z_n) = \left(\sum^n_{k=1} v(z_k)\, \frac{\partial}{\partial z_k} + \sum^n_{k=1} \overline{v(z_k)}\, \frac{\bar\partial}{\bar\partial z_k} + \lambda_jv'(z_j) + \lambda_{*j}\overline{v'(z_j)}\right) f(z_1, \ldots, z_n).$$

If $f$ is a pre-pre-Schwarzian form of order $\mu$, i.e., $f_D = f_{\widetilde D}\circ h + \mu \log h',$ then 
$$f_t(z) =  f(\psi_t(z)) + \mu \log \psi_t'(z),$$
and $\mathcal{L}_v f$ becomes
$$\mathcal{L}_{v} f = v \partial  f + \mu v'.$$

Let $D$ be a simply connected domain in $\mathbb{C}$ and consider the stochastic flow corresponding to 
\begin{equation}
 \label{eq:shsf2}
\begin{cases}
 d w_t(z) = -b(w_t(z)) \, dt + \sqrt{\kappa}\, \sigma(w_t(z)) \circ dB_t,\\
 w_0(z) = z,
\end{cases}z\in D, \kappa>0,
\end{equation}
with holomorphic coefficients $b, \,\sigma :D \to \mathbb{C}$. Let $f $ be a (smooth) non-random conformal field in $D.$ We can write It\^o's formula in terms of the Lie derivatives as
\begin{align}
\label{eq:itolie}
 df(w_t(z)) 
= \left(-\mathcal{L}_b +\frac{\kappa}{2} \mathcal{L}^2_{\sigma}\right) f (w_t(z)) 
+ \sqrt{\kappa}  \, \mathcal{L}_{\sigma} f (w_t(z)) \,dB_t,
\end{align}
so that the infinitesimal generator of the diffusion \eqref{eq:shsf2} is expressed as
\begin{equation}
\label{eq:lieinfgen}
\mathcal{A} = - \mathcal{L}_{b} + \frac{\kappa}{2} \mathcal{L}^2_{\sigma}, 
\end{equation}
and the quadratic covariation is given by
\begin{equation}
\label{eq:liequadcov}
 d \langle f(w_t(z_1)), \, f(w_t(z_2))\rangle = \kappa \, \mathcal{L}_{\sigma} f(w_t(z_1))\,  \mathcal{L}_{\sigma} f(w_t(z_2)) \, dt, \quad {z_1, \, z_2\in D}.
\end{equation}

Let $\ell_{n} = - z^{n+1},$ $n=-2,\ldots 1,$ and let $G_{\mathbb{H}}(z_1,z_2)$ be Green's function in $\mathbb{H}$,
\[
 G_{\mathbb{H}}(z_1,z_2) = \log\left|\frac{z_1 - \overline{z_2}}{z_1 - z_2}\right|.
\]
Then,
\begin{align*}
 \mathcal{L}_{\ell_n} G_{\mathbb{H}}(z_1,z_2) 
=\re\left[\frac{z^{n+1}_1 -z^{n+1}_2}{z_1-z_2} - \frac{z^{n+1}_1 -\overline{z^{n+1}_2}}{z_1 - \overline{z_2}}\right].
\end{align*}
In particular, if $b$ and $\sigma$ have the form  \eqref{eq:bformula} and \eqref{eq:sigmaformula} respectively, then
\begin{equation}
\label{eq:liehadamard}
 \mathcal{L}_b G_{\mathbb{H}}(z_1,z_2) = 4 \im \frac{1}{z_1} \, \im \frac{1}{z_2},
\end{equation}
and
\begin{equation}
\label{eq:lsigma}
 \mathcal{L}_{\sigma} G_{\mathbb{H}}(z_1,z_2) = 0.
\end{equation}
In fact, the vector field $\sigma$ generates a flow of M\"obius automorphisms of $\mathbb{H}$, and therefore, \eqref{eq:lsigma} is a consequence of the invariance of $G_{\mathbb{H}}(z_1,z_2)$ with respect to such automorphisms.
Equation \eqref{eq:liehadamard} is a special case of Hadamard's formula, see \cite{neharibook}.

\section{Coupling of slit holomorphic stochastic flows and GFF}

\subsection{Coupling of GFF and chordal SLE}
Let $X$ and $Y$ be two random variables, which can in principle be defined on two different probability spaces. A coupling of $X$ and $Y$ is a new probability space together with random variables $X'$ and $Y'$, such that $X'$ has the same law as $X$, and $Y'$ has the same law as $Y$. This definition can be easily extended to more general random objects. In particular, we study couplings of stochastic processes (namely, slit holomorphic stochastic flows) with  random distributions (modifications of GFF).

Let $\tilde{\Phi}_{\mathbb{H}}$ be a modification of the GFF in $\mathbb{H}$, and let $\{w_t\}_{t\geq 0}$ be a slit holomorphic stochastic flow in $\mathbb{H}$ with $\kappa \in (0,4]$. \todo{it will only work for $\kappa =4$.}
Note that the Lebesgue measure of $\mathbb{H}\setminus w_t^{-1}(\mathbb{H})$ is almost surely zero for all $t\geq0$.

Let $\mathcal{F}_{T}$ denote the $\sigma$-algebra $\sigma\left\{w_s: s\in [0,T]\right\}$. We are interested in such couplings of $\{w_t\}_{t\geq 0}$ and $\tilde{\Phi}_{\mathbb{H}}$ that have the following property 
\begin{equation}
 \label{eq:couplingcond}
\mathsf{Law}\left(\tilde{\Phi}_{\mathbb{H}} \, \middle| \, \mathcal{F}_{T}\right) =  \mathsf{Law}(\tilde{\Phi}_{\mathbb{H}} \circ w_T), \quad \textrm{ for all } T > 0.
\end{equation}
Thus, for $\mathfrak{p}_1, \mathfrak{p}_{2}\in C^{\infty}_0(\mathbb{H})$,
\[
 \mathsf{E}((\tilde{\Phi}_{\mathbb{H}},\mathfrak{p}_1)|\mathcal{F}_T ) = \int_{H_T} u(w_T(z))\, \mathfrak{p}_1(z)\, dA(z),
\]
and
\begin{align*}
 \mathsf{Cov} \left((\tilde{\Phi}_{\mathbb{H}}, \mathfrak{p}_1), \,(\tilde{\Phi}_{\mathbb{H}}, \mathfrak{p}_2)\,\middle| \mathcal{F}_T\,\right) = \int_{H_T \times H_T} 2 \,G_{H_T} (z_1, z_2)\, \mathfrak{p}(z_1)\, \mathfrak{p}_2(z_2)\, dA(z_1)\, dA(z_2).
\end{align*}

This property implies, in particular, that the curves generated by $\{w_t\}_{t\geq0}$ can be viewed as  `level lines' of $\tilde{\Phi}_{\mathbb{H}}$, see \cite{sheffieldcontour1}, \cite{sheffieldcontour2}. \todo{Add $\Psi$, similar to Kang, Makarov}
For example, chordal SLE  can be coupled with the GFF as follows, see \cite{ dub09,sheffieldcontour1, sheffieldcontour2}.

Let $\{w_t\}_{t\geq 0}$ stand for the stochastic flow  of the chordal SLE$(\kappa)$
\[
\begin{cases}
 dw_t(z) = \dfrac{2}{w_t(z)}\, dt - \sqrt{\kappa}\, dB_t,\\
w_0(z) = z,
\end{cases}z\in \mathbb{H}.
\]
Let $\Phi_{H_t}$ be a GFF with zero boundary conditions on $H_t$, independent of $B_t$,
and let 
\[
u_t(z) = \sqrt{\frac{8}{\kappa}}\arg w_t(z) + \left(\sqrt{\frac{8}{\kappa}} - \sqrt{\frac{\kappa}{2}}\right)\, \arg w'_t(z).
\]
Let
\[
 \tilde{\Phi}_{H_t}= \Phi_{H_t} + u_t.
\]
Then for every real constant $T\geq 0$, given $\mathcal{F}_T$, the laws of the random distributions $\tilde{\Phi}_{\mathbb{H}}(\omega)$  and those of  $\tilde{\Phi}_{H_T (\omega)}(\omega)$ coincide, hence \eqref{eq:couplingcond} holds.
Couplings of radial, chordal, dipolar and annulus SLE with various modifications of the GFF were studied in \cite{izyurov}.

\subsection{Hadamard's formula and coupling}

The following proposition states that a key ingredient for the existence of a coupling of a slit holomorphic stochastic flow and a modification of GFF $\tilde{\Phi}_{H_t} = \Phi_{H_t} + u_t$ that satisfies \eqref{eq:couplingcond} is the relation
\[
 dG_{H_t}(z_1,z_2) = -\frac12 d \langle (u_t(z_1), u_t(z_2)) \rangle.
\]
This equation is sometimes also called Hadamard's formula because it describes variation of Green's function, similarly to the formula derived in \cite{hadamard1908memoire}.

The proof of the proposition is mostly an adaptation from \cite{sheffieldcontour1, sheffieldcontour2}.

\begin{proposition}
\label{th:thcoupling}
Let $\{w_t\}_{t\geq0}$ be a normalized slit holomorphic stochastic flow in $\mathbb{H}$ 
\begin{equation}
\begin{cases}
 dw_t(z) = -b(w_t(z)) \, dt + \sqrt{\kappa}\, \sigma(w_t(z))\circ dB_t,\\
w_0(z) = z, \quad z\in \mathbb{H}
\end{cases}
\end{equation}
with $\kappa \in(0,4]$. 
Suppose $u$ is a positive harmonic function in $\mathbb{H}$ defining an imaginary part of a pre-pre-Schwarzian form of order $-2\mathfrak{b}$ by the rule 
\[
 u_t(z) = u(w_t(z)) -2\mathfrak{b} \arg w'_t(z)
\]
such that 
\[
  \left(- \mathcal{L}_{b} + \frac{\kappa}{2} \mathcal{L}^2_{\sigma} \right) u  (z) = 0, \quad \textrm{for all } z\in \mathbb{H},
\]
and suppose $|u_t|$ is bounded by a continuous function not depending on $t$.

Let $\Phi_{\mathbb{H}}$ be a Gaussian free field independent of the process $B_t$.  Denote $H_t := w^{-1}_t(\mathbb{H})$, $\tilde{\Phi}_{\mathbb{H}} := \Phi_{\mathbb{H}} + u$, $\tilde{\Phi}_{H_t} := \tilde{\Phi}_{\mathbb{H}} \circ w_t = \Phi_{H_t} + u_t$, $t\geq 0$.

 Suppose that
\begin{equation}
\label{eq:hadassumption}
 dG_{H_t}(z_1,z_2) = -\frac12 d \langle u_t(z_1), u_t(z_2) \rangle
\end{equation}
holds for all $t\geq 0$ with probability 1. Then for every $T\geq 0$, given $\mathcal{F}_T$, the random distributions $\tilde{\Phi}_{\mathbb{H}}(\omega)$  and $\tilde{\Phi}_{H_T(\omega)}(\omega)$ have the same law, hence $\tilde{\Phi}_{H_T(\omega)}(\omega)$ gives a coupling of GFF with $\{w_t\}_{t\geq 0}$ that satisfies \eqref{eq:couplingcond}.

\end{proposition}

The following two auxiliary results are used in the proof of the proposition.

\begin{lemma}[{Conditional Fubini theorem, see e.g. \cite[{Theorem 1.1.8}]{Applebaum}}]
Let $(\Omega, \mathcal{F}, \mathbb{P})$ be a probability space, and let $(S, \Sigma, \mu)$ be a $\sigma$-finite measure space. If
\[
F\in L^1(S\times \Omega, \Sigma \otimes \mathcal{F}, \mu \times \mathbb{P}),
 \]
then
\[
\mathsf{E}\left(\left|\int_S \mathsf{E} \left[F(s,\omega)\,|\,\mathcal{G}\right]\, d\mu(s)\right|\right) < \infty,
\]
and
\[
\mathsf{E}\left[\int_S  F(s,\cdot)\, d\mu(s)\,|\, \mathcal{G}\right] = \int_S \mathsf{E} \left[F(s,\cdot)\,|\,\mathcal{G}\right] d\mu(s) \quad \textrm{a.s.}
\]

\end{lemma}

\begin{lemma}[weighted averages of martingales]
\label{lem:weightedavg}
Let for each $z\in D$, $u_t(z)$ be a continuous martingale with respect to the filtration $\{\mathcal{F}_t\}_{t\geq 0}$, and let for each $t\geq 0$, $\mathsf{E} |u_t(z)| \in L^{1}_{loc}(D)$ as a function of $z$.  Then, for every $\mathfrak{p}\in C^{\infty}_0(D)$, the weighted average
\[
 M_t(\omega) := \int_D u_t(z) \,\mathfrak{p}(z) \,dA(z)
\]
is a martingale.

For a given $\omega_0\in \Omega,$ $t\mapsto M_t(\omega_0)$ is continuous at $t_0$ if there exists a function $U\in L^1_{loc}(D)$ such that
\[
 |u_t(z,\omega_0)| \leq U(z)\quad \textrm{ for all } t\in(t_0-\epsilon,t_0+\epsilon),
\]
or, alternatively, that $u_t(z)$ is continuous at $t_0$ in the local uniform topology.
\end{lemma}
\begin{proof}
 
Let $\mathfrak{p} \in C^{\infty}_0(D)$. Then $\mathsf{E}|u_t(z) \, \mathfrak{p}(z)| < \infty$, so that $u_t(z) \, \mathfrak{p}(z) \in L^{1}(D \times \Omega, A \times \mathbb{P})$, and by the conditional Fubini theorem, for $0\leq s\leq t < \infty$,
\[
 \mathsf{E} \left[M_t | \mathcal{F}_s\right] = \int_D \mathsf{E} \left[u_t(z)\,| \,\mathcal{F}_s \right] \, \mathfrak{p}(z) \, dA(z) = M_s \quad\textrm{a.s.},
\]
and
\[
 \mathsf{E} |M_t| \leq  \int_D \mathsf{E} |u_t(z)| \,|\mathfrak{p}(z)|\, dA(z) < \infty.
\]

If $u_t(z)$ is bounded by an $L^1_{loc}(D)$-function $U(z)$ for $t\in(t_0-\epsilon, t_0+\epsilon)$, then  continuity at $t_0$ follows from the dominated convergence theorem.
\end{proof}

\begin{proof}[Proof of Proposition~\ref{th:thcoupling}]
For every fixed $z\in \mathbb{H}$, $w_t(z)$ is defined for all $t\in[0,+\infty)$ with probability $1$. On the remaining event of probability zero we can define $w_t(z)$ arbitrarily. 

The process $u_t(z)$ is a continuous martingale. Indeed,  it is a continuous local martingale by the It\^o calculus, moreover, it is bounded, hence a martingale. Since $u_t(z)$ satisfies the assumptions of Lemma~\ref{lem:weightedavg},  the process $(u_t,\mathfrak{p})$ is a continuous martingale for every $\mathfrak{p}\in C^{\infty}_0(\mathbb{H})$.

The quadratic variation of the continuous local martingale $(u_t,\mathfrak{p})$ is a unique adapted continuous process $\langle (u_t,\mathfrak{p})\rangle$ of bounded variation satisfying
\begin{equation}
\label{eq:qv1}
 \langle (u_0,\mathfrak{p})\rangle = 0\quad \textrm{a.s.},
\end{equation}
and such that 
\begin{equation}
\label{eq:qv2}
 (u_t,\mathfrak{p})^2-\langle (u_t,\mathfrak{p})\rangle
\end{equation}
is a continuous local martingale, see e.g., \cite[Problem 5.17]{karatzas}.

Consider
\[
 M_t(z_1, z_2) := u_t(z_1) \, u_t(z_2) + 2\,G_{H_t}(z_1, z_2), \quad z_1, z_2\in \mathbb{H}.
\]
By the It\^o calculus and~\eqref{eq:hadassumption}, $M_t(z_1, z_2)$ is a continuous local martingale for each pair $z_1,z_2 \in \mathbb{H}$, $z_1\neq z_2$. Moreover, $u_t(z_1)u_t(z_2)$ is bounded uniformly for all $t\geq 0$, and $G_{H_t}(z_1, z_2)$ is non-increasing by monotonicity of Green's function. Hence, $M_t(z_1,z_2)$ is uniformly bounded, and is in fact, a martingale for each $z_1, z_2\in \mathbb{H},$ $z_1\neq z_2$. 

We can apply Lemma \ref{lem:weightedavg} twice, and see that
\begin{align*}
 \int_{\mathbb{H}\times \mathbb{H}} &M_t(z_1, z_2) \, \mathfrak{p}(z_1)\,\mathfrak{p}(z_2)\, dA(z_1)\, dA(z_2) \\
&= (u_t, \mathfrak{p})^2 + \int_{\mathbb{H}\times \mathbb{H}} 2\,G_{H_t}(z_1, z_2) \, \mathfrak{p}(z_1)\, \mathfrak{p}(z_2)\,dA(z_1)\, dA(z_2)\\
&= (u_t, \mathfrak{p})^2 + \|\mathfrak{p}\|^2_{\mathcal{E}(H_t)}.
\end{align*}
is a continuous martingale. 

Let us denote $E_t = \|\mathfrak{p}\|^2_{\mathcal{E}(H_t)}$. By monotonicity of Green's function, the process $E_0 - E_t$ is monotonically increasing, and hence, is of bounded variation. Trivially, it satisfies \eqref{eq:qv1}, and by the argument above, it also satisfies \eqref{eq:qv2}, therefore, we can conclude that
\[
 \langle(u_t,\mathfrak{p})\rangle  = E_0-E_t.
\]

Let $\tau_t:= \inf\{s:E_0 - E_{s} = t\}$. By L\'evy's characterization of Brownian motion,  the process $(u_{\tau_t},\mathfrak{p})$ is a standard Brownian motion started at $(u, \mathfrak{p})$. Therefore, $(u_{\tau_t},\mathfrak{p}) \sim \mathcal{N}((u, \mathfrak{p}), t)$. Consequently,
for $T\geq 0,$ the conditional law of $(u_T,\mathfrak{p})$ with respect to the sigma-algebra $\mathcal{G} = \sigma\{E_s:0\leq s \leq T\}$ is $\mathcal{N}((u, \mathfrak{p}), E_0 - E_T)$. 

Now, given $\mathcal{G}$, the random variables $(\Phi_{H_T}, \mathfrak{p})$ and $(u_T,\mathfrak{p})$ are conditionally independent for $\mathfrak{p}\in C^{\infty}_0(\mathbb{H})$ . Given $\mathcal{G}$, their conditional laws  are 
\[
 (\Phi_{H_T}, \mathfrak{p}) \sim \mathcal{N}(0, E_T), \quad  (u_{H_T}, \mathfrak{p}) \sim  \mathcal{N}((u, \mathfrak{p}), E_0 - E_T),
\]
hence their sum has the law
\[
 (\tilde{\Phi}_{H_T}, \mathfrak{p}) \sim \mathcal{N}((u, \mathfrak{p}),E_0).
\]
Thus, $\tilde{\Phi}_{H_T}$ has the same law as $\tilde{\Phi}_{\mathbb{H}}$, and the proposition is proved.
\end{proof}

\section{Slit holomorphic flows that admit coupling with GFF}

It is natural to ask for which slit holomorphic stochastic flows it is possible to find a harmonic function  $u$ satisfying the assumptions of Proposition~\ref{th:thcoupling}.

\subsection{Existence of $u$}

\begin{lemma}
Let $\{w_t\}_{t\geq0}$ be a normalized slit holomorphic stochastic flow in $\mathbb{H}$ 
\begin{equation}
\begin{cases}
 dw_t(z) = -b(w_t(z)) \, dt + \sqrt{\kappa}\, \sigma(w_t(z))\circ dB_t,\\
w_0(z) = z, \quad z\in \mathbb{H}.
\end{cases}
\end{equation}

There exists a harmonic function $u$ as an imaginary part of a pre-pre-Schwarzian form $u_t$ of order $-2\mathfrak{b}$ defined as in Proposition \ref{th:thcoupling} satisfying the conditions
\begin{equation}
\label{eq:hannihilates}
  \left(- \mathcal{L}_{b} + \frac{\kappa}{2} \mathcal{L}^2_{\sigma} \right) u_0  (z) = 0, \quad \textrm{for all } z\in \mathbb{H}
\end{equation}
and
\begin{equation}
\label{eq:hadlem3}
 dG_{H_t}(z_1,z_2) = - \frac12 d \langle (u_t(z_1), u_t(z_2)) \rangle
\end{equation}
if and only if there exist some constants $\alpha, \beta \in \mathbb{R}$ such that 
\begin{equation}
\label{eq:bsigmarel}
b(z) =  \frac{- \sqrt{\kappa} \sigma(z) \,(\sqrt{\kappa} \sigma(z) + \beta z^2)  -\mathfrak{b}\, \sqrt{2\kappa }\,z^2 \, (b'(z) \, \sigma(z) - b(z)\, \sigma'(z))}{z \,(\alpha z + 2)}
\end{equation}
for all $z\in \mathbb{H}.$ 

In that case, $u$ is given by
\begin{equation}
\label{eq:1observable}
 u(z) = -\sqrt{\frac{2}{\kappa}}\, \im \int \frac{2+\alpha \,z}{z \,\sigma(z)}\,dz + 2\mathfrak{b} \arg \sigma(z).
\end{equation}

\end{lemma}
\begin{proof}
Suppose that such a function $u$ exists. 
By \eqref{eq:lieinfgen}, \eqref{eq:liehadamard} and \eqref{eq:lsigma}, we have that
\begin{align*}
 dG_{H_t}(z_1,z_2) 
&= \left(- \mathcal{L}_{b} + \frac{\kappa}{2} \mathcal{L}^2_{\sigma}\right) G_{\mathbb{H}}(w_t(z_1), w_t(z_2)) \,dt + \sqrt{\kappa}  \, \mathcal{L}_{\sigma} G_{\mathbb{H}} (w_t(z_1), w_t(z_2)) \circ dB_t \\
&= - 4 \im \frac{1}{w_t(z_1)}\, \im \frac{1}{w_t(z_2)}.
\end{align*}
By \eqref{eq:liequadcov},
\[
- \frac12 d \langle (u_t(z_1), u_t(z_2)) \rangle =- \frac{\kappa}{2}\, \mathcal{L}_{\sigma} u(w_t(z_1))\, \mathcal{L}_{\sigma} u(w_t(z_2)),
\]
and we rewrite \eqref{eq:hadlem3} as
\[
 -4\,\im \frac{1}{w_t(z_1)}\, \im \frac{1}{w_t(z_2)} = - \frac{\kappa}{2}\, \mathcal{L}_{\sigma} u(w_t(z_1))\, \mathcal{L}_{\sigma} u(w_t(z_2)).
\]
Then in the upper half-plane uniformization, we have 
\[
\pm\im \frac{2}{z} = \sqrt{\frac{\kappa}{2}} \,\mathcal{L}_{\sigma} u (z),
\]
where the sign choice depends on normalization of the vector field $b$. 
\todo{Explain this}
In the forward case we choose 
\[
-\im \frac{2}{z} = \sqrt{\frac{\kappa}{2}} \,\mathcal{L}_{\sigma} u (z).
\]

Let $u_{+}$ be a holomorphic pre-pre-Schwarzian form of order $-2\mathfrak{b}$ such  that $u = \im u^{+}$. Then in the upper half-plane uniformization
\[
-\im \frac{2}{z} = \sqrt{\frac{\kappa}{2}} \,\im \mathcal{L}_{\sigma} u^{+} (z),
\]
and
\[
 -\frac{2}{z} = \sqrt{\frac{\kappa}{2}} \, \left(\sigma(z) ({u^{+}})' (z) -2\mathfrak{b}\, \sigma'(z)\right) + \alpha,
\]
for some real constant $\alpha$, and \eqref{eq:1observable} follows.

Now, to make sure that \eqref{eq:hannihilates} is satisfied, we calculate
\[
 \mathcal{L}^2_{\sigma} \,u (z) = 2 \,\sqrt{\frac{2}{\kappa}} \, \im \frac{\sigma(z)}{z^2}.
\]
Here, we use the fact that Lie derivative of a pre-pre-Schwarzin form is a $(0,0)$-differential.
Also we calculate 
\[
 \mathcal{L}_{b} u (z) = -\sqrt{\frac{2}{\kappa}}\,\im \left[b(z)\, \frac{2 + \alpha z}{z \, \sigma(z)}\right] +2\mathfrak{b} \, \im\left[\frac{b(z) \,\sigma'(z) - b'(z)\, \sigma(z)}{\sigma(z)}\right]. 
\]
Thus we have 
\begin{align*}
\left(- \mathcal{L}_{b} + \frac{\kappa}{2} \mathcal{L}^2_{\sigma} \right) u  (z) 
=  \im \left[ \sqrt{\frac{2}{\kappa}}\, b(z) \frac{2+\alpha z}{z \, \sigma(z)} + \sqrt{2\, \kappa}\, \frac{\sigma(z)}{z^2} +2\mathfrak{b}\, \frac{b'(z) \, \sigma(z) - b(z)\, \sigma'(z)}{\sigma(z)}\right],
\end{align*}
and  condition \eqref{eq:hannihilates} can be reformulated as
\[
  \sqrt{\frac{2}{\kappa}}\, b(z) \frac{2+\alpha z}{z \, \sigma(z)} + \sqrt{2\, \kappa}\, \frac{\sigma(z)}{z^2}  +2\mathfrak{b}\, \frac{b'(z) \, \sigma(z) - b(z)\, \sigma'(z)}{\sigma(z)}= -\sqrt{2} \beta, 
\]
for some real constant $\beta$, and \eqref{eq:bsigmarel} follows.

\end{proof}

\begin{lemma}
\label{lem:linsystem}
 Let \begin{equation}
\label{eq:bsigmarepr}
 b(z) = -\left(\frac{2}{z} + b_{-1} +b_0 z + b_1 z^2\right),\quad \sigma(z) = -(1 + \sigma_0 z + \sigma_1 z^2).
\end{equation}
Then  relation~\eqref{eq:bsigmarel} holds for all $z\in \mathbb{H}$ if and only if 
$\mathfrak{b} =\sqrt{\kappa/8} - \sqrt{2/\kappa}$ and 
\begin{equation}
\label{eq:linsystem}
 \begin{cases}
b_{-1}=-\alpha+4\sigma_0,\\
(2\alpha+(\kappa-4)\sigma_0)b_{-1}-(\kappa-8)b_0
=-2\beta\sqrt{\kappa}+2\kappa\sigma_0^2-2\kappa\sigma_1+24\sigma_1,\\
(\kappa-4)\sigma_1b_{-1}+\alpha b_0-(\kappa-6)b_1=-\beta\sqrt{\kappa}\sigma_0+2\kappa
\sigma_0\sigma_1,\\
(\kappa-4)\sigma_1b_0-((\kappa-4)\sigma_0-2\alpha)b_1=(-2\beta\sqrt{\kappa}+2\kappa\sigma_1)\sigma_1.
\end{cases}
\end{equation}
\end{lemma}
\begin{proof}
After substituting \eqref{eq:bsigmarepr} in \eqref{eq:bsigmarel} and multiplying both sides by $z\,(\alpha z + 2)$, we equate the coefficients at powers of $z$ and arrive at the relations above.
\end{proof}

Let $a = \sqrt{2/\kappa}$ and $\mathfrak{b} =\sqrt{\kappa/8} - \sqrt{2/\kappa}.$

\begin{proposition}
\label{th:classification}
Up to a conjugation with a conformal isomorphism, there are only the following four slit holomorphic stochastic flows in $\mathbb{H}$  \todo{The ``only if'' part should also be proved}  admitting the coupling with the Dirichlet modifications of the GFF satisfying \eqref{eq:couplingcond}:

\begin{enumerate}
\item  \label{th:classification chordal} Chordal {\rm SLE} with drift $\alpha t$, $\alpha \in \mathbb{R}$
\[
 b(z)  = -\frac{2}{z} + \alpha, \quad \sigma(z) = -1,
\]
\[
 u(z) = 2a\, \arg z + \alpha a \, \im z;
\]
\item \label{th:classification chordal2} The parabolic ($\sigma(z) = -1$) family with $\alpha=0$
\[
b(z)  = -\frac{2}{z} - \,\frac{2 \,\beta\,\sqrt{\kappa}}{\kappa - 8}\, z, \quad \beta\in \mathbb{R}, 
\]
\[
 u(z) = 2a\, \arg z;
\]

 \item \label{th:classification dipolar} Dipolar {\rm SLE} with drift  $\alpha t,$ $\alpha \in \mathbb{R}$
\[ 
b(z) =-\frac{2}{z} + \alpha + \frac{z}{2}- \frac{\alpha}{4} \, z^2,
\quad \sigma(z) = - 1+ \frac14 z^2, 
\]
\[
u(z) = 2a  \arg z +(2\mathfrak{b}-a(1+\alpha)) \arg(2-z)+(2\mathfrak{b}-a(1-\alpha)) \arg(2+z) ;
\]
\item \label{th:classification dipolar2} The hyperbolic ($\sigma(z) = - 1+ \frac14 z^2$) family with $\alpha = \pm(\kappa-6)/2$
\[
b(z) = -\frac{2}{z} \pm \frac{\kappa-6}{2} -\left(\frac{3-\kappa}{2} +\frac{2\beta \sqrt{\kappa}}{\kappa-8}\right) z \pm \frac18\,\left(\kappa-2 - \frac{8\,\beta\,\sqrt{\kappa}}{\kappa-8}\right) z^2, \quad\beta \in \mathbb{R}
\]
\[ 
u(z) = (\kappa - 6) a\, \arg(2\pm z)+ 2a\,\arg z.
\]
\item \label{th:classification radial} Radial {\rm SLE}, $\kappa=6$, with drift  $\alpha t,$ $\alpha \in \mathbb{R}$, 
$$ b(z) = -\frac{2}{z} + \alpha - \frac{z}{2} +\alpha\,\frac{z^2}{4},\quad \sigma(z) = - 1- \frac14 z^2,$$
$$u(z)=-  \alpha \,\sqrt{\frac{1}{3}}\, \log\left|\frac{z-2i}{z+2i}\right|+ \sqrt{\frac{4}{3}} \arg z.$$
\end{enumerate}
 
\end{proposition}
\begin{proof}
We want to find all possible combinations of $b$ and $\sigma$ such that \eqref{eq:bsigmarel} is satisfied.

We identify flows that differ by a conjugation with an automorphism of $\mathbb{H}$. Let $\{w_t\}_{t\geq 0}$ be the flow of
\[
\begin{cases}
 dw_t(z) = -b(w_t(z)) \, dt + \sqrt{\kappa}\, \sigma(w_t(z))\circ dB_t,\\
w_0(z) = z, \quad z\in \mathbb{H}.
\end{cases}
\]
If $\phi:\mathbb{H}\to \mathbb{H}$ is a conformal automorphism of the upper half-plane, then the flow $\{\tilde{w}_{t}:= \phi\circ w_t \circ \phi^{-1}\}_{t\geq 0}$ satisfies
\[
\begin{cases}
 d\tilde{w}_t(z) = -\phi_ *b(\tilde{w}_t(z)) \, dt + \sqrt{\kappa}\, \phi_* \sigma(\tilde{w}_t(z))\circ dB_t,\\
\tilde{w}_0(z) = z, \quad z\in \mathbb{H},
\end{cases}
\]
with $\phi_*$ denoting the push-forward of vector fields, e.g., $$\phi_* \sigma (z) = \frac{1}{(\phi^{-1})'(z)} \, \sigma (\phi^{-1}(z)).$$ 

We can define an equivalence relation on the set of normalized complete fields by letting $\sigma \sim \widetilde\sigma$ if there exists an automorphism $\phi:\mathbb{H} \to \mathbb{H}$ such that $\sigma = \phi_* \widetilde\sigma$. This equivalence relation partitions the set of $\sigma$'s into three equivalence classes, and it suffices to work with one representative in each class:
\renewcommand{\theenumi}{\roman{enumi}}
\begin{enumerate}
\medskip\item $\sigma(z) = -1$ (parabolic, as in the chordal SLE),
\medskip\item $\sigma(z) = -1 + \frac14 z^2$ (hyperbolic, as in the dipolar SLE).
\medskip\item $\sigma(z) = -1 - \frac14 z^2$ (elliptic, as in the radial SLE),
\end{enumerate}
\renewcommand{\theenumi}{\alph{enumi}}
\medskip

Next, for each of these three vector fields $\sigma$ we apply Lemma \ref{lem:linsystem} to determine vector fields $b$ such that \eqref{eq:bsigmarel} is satisfied. We simultaneously calculate the harmonic functions $u$ using \eqref{eq:1observable}.
For each of these three vector fields $\sigma,$ we have $\sigma_0=0$ and therefore, the system \eqref{eq:linsystem} becomes 
$b_{-1}=-\alpha$ and 
\begin{equation}
\label{eq:linsystem1}
 \begin{cases}
(\kappa-8)b_0=-2\alpha^2+2\beta\sqrt{\kappa}+(2\kappa-24)\sigma_1,\\
\alpha b_0-(\kappa-6)b_1=\sigma_1(\kappa-4)\alpha,\\
(\kappa-4)\sigma_1b_0+2\alpha b_1=(-2\beta\sqrt{\kappa}+2\kappa\sigma_1)\sigma_1.
\end{cases}
\end{equation}

In the \textbf{parabolic} case, we have $\sigma_1=0$ and therefore, the system \eqref{eq:linsystem1} becomes 
$$
 \begin{cases}
(\kappa-8)b_0=-2\alpha^2+2\beta\sqrt{\kappa},\\
\alpha b_0-(\kappa-6)b_1=0,\\
\alpha b_1=0.
\end{cases}
$$

\noindent (Case 1) If $\alpha\ne 0,$  then $b_0=b_1=0$ and $\beta = \alpha^2/\sqrt\kappa.$ 
This solution gives \eqref{th:classification chordal}. 

\noindent (Case 2) If $\alpha = 0,$ then $b_0 = 2\beta\sqrt\kappa/(\kappa-8)$ and $b_1=0$ unless $\kappa=6$ or $8.$
This solution gives \eqref{th:classification chordal2}.
When $\kappa =6,$ then $b_1$ can be arbitrary. 
When $\kappa=8$ and $\beta = 0,$ then $b_0$ can be arbitrary.

In the \textbf{hyperbolic} case, we have $\sigma_1=-1/4$ and therefore, the system \eqref{eq:linsystem1} becomes 
$$
\begin{cases}
(\kappa-8)b_0=-2\alpha^2+2\beta\sqrt{\kappa}-\kappa/2+6,\\
\alpha b_0-(\kappa-6)b_1=(1-\kappa/4)\alpha,\\
(1-\kappa/4)b_0+2\alpha b_1=\beta\sqrt{\kappa}/2+\kappa/8.
\end{cases}
$$
If the solution $(b_0,b_1)$ to the first two linear equations is substituted into the last one, it becomes
$$\frac{(1-\alpha^2+\beta\sqrt\kappa)(4\alpha^2-(\kappa-6)^2)}{(\kappa-8)(\kappa-6)}=0.$$

\noindent (Case 3) If $1-\alpha^2+\beta\sqrt\kappa = 0,$ then we have $(b_0,b_1) = (-1/2,\alpha/4).$ This solution gives \eqref{th:classification dipolar}.

\noindent (Case 4) If $\alpha = \pm(\kappa-6)/2,$ then we have 
$$b_0 =   \frac{3-\kappa}{2} +\frac{2\beta \sqrt{\kappa}}{\kappa-8},
\qquad b_1 = \mp \frac18\,\left(\kappa-2 - \frac{8\,\beta\,\sqrt{\kappa}}{\kappa-8}\right).$$
This solution gives \eqref{th:classification dipolar2}.

In the \textbf{elliptic} case, we have $\sigma_1=1/4$ and therefore, the system \eqref{eq:linsystem1} becomes 
$$
\begin{cases}
(\kappa-8)b_0=-2\alpha^2+2\beta\sqrt{\kappa}+\kappa/2-6,\\
\alpha b_0-(\kappa-6)b_1=(\kappa/4-1)\alpha,\\
(\kappa/4-1)b_0+2\alpha b_1=-\beta\sqrt{\kappa}/2+\kappa/8.
\end{cases}
$$
If the solution $(b_0,b_1)$ to the first two linear equations is substituted into the last one, it becomes
$$\frac{(1+\alpha^2-\beta\sqrt\kappa)(4\alpha^2+(\kappa-6)^2)}{(\kappa-8)(\kappa-6)}=0.$$

\noindent (Case 5) Since $\alpha$ is real, we have $\beta =(1+\alpha^2)/\sqrt\kappa.$ 
Therefore, there is only one solution 
$$ b(z) = -\frac{2}{z} + \alpha - \frac{z}{2} +\alpha\,\frac{z^2}{4}.$$
In this case, 
\begin{eqnarray*}
u(z) &=&  -  \alpha \,\sqrt{\frac{8}{\kappa}}\,\im \arctan \frac{2}{z} + \sqrt{\frac{8}{\kappa}} \arg z + \frac{\kappa - 6}{\sqrt{2 \kappa}}\arg(4+z^2)\\
       &=&  -  \alpha \,\sqrt{\frac{2}{\kappa}}\, \log\left|\frac{z-2i}{z+2i}\right|+ \sqrt{\frac{8}{\kappa}} \arg z + \frac{\kappa - 6}{\sqrt{2 \kappa}}\arg(4+z^2).
\end{eqnarray*}
However, the branch of $\arg(4+z^2)$ cannot be chosen in such a way that $u$ is continuous in the whole upper half-plane unless $\kappa=6$, therefore, we cannot apply Proposition \ref{th:thcoupling}.
The  choice $\kappa=6$ leads to the radial SLE(6) with drift \eqref{th:classification radial}.

To ensure that $u_t(z)$ in \eqref{th:classification chordal} is bounded by a continuous function uniformly with respect to $t$, we note that 
\[
 d \im w_t(z) = - 2\, \frac{\im w_t(z)}{|w_t(z)|^2}\,dt,
\]
that is, $\im w_t(z)$ is a decreasing function. Thus, $|u_t(z)| \leq \alpha a\,\im z + 2a\pi$, and we can apply Proposition \ref{th:thcoupling}.

In each of the cases (\ref{th:classification chordal2}--\ref{th:classification radial}), the function $u_t(z)$ is bounded by a constant, so that Proposition \ref{th:thcoupling} can be applied as well.
\end{proof}

\begin{remark*}
The cases (\ref{th:classification chordal}), (\ref{th:classification dipolar}), and (\ref{th:classification radial}) correspond, respectively, to the chordal, dipolar, and radial $\mathrm{SLE}_\alpha(\kappa)$ (in the last case $\kappa=6$) driven by $\sqrt\kappa B_t + \alpha \,t$  (Brownian motion with drift). 
In the next two sections, we implement conformal field theory of chordal and dipolar  $\mathrm{SLE}_\alpha(\kappa).$

The case (\ref{th:classification chordal2}) was mentioned in \cite{ivanov2014general}, and when $\kappa = 4$ the curves generated by this flow have the same law as the chordal SLE(4) stopped at the time $1/(4\beta)$.

\end{remark*}

\section{CFT of chordal SLE with drift}

In this section, we study certain martingale-observables of the chordal SLE with drift, see item \eqref{th:classification chordal}, Proposition~\ref{th:classification}, by means of a version of conformal field theory
with the Dirichlet boundary conditions in which the Fock space  fields are constructed by a special background charge modification of GFF $\Phi=\Phi_{(0)}$ with the Dirichlet boundary conditions.

\subsection{Modifications of Gaussian free field}
Given a simply connected hyperbolic domain $D$ with a marked point $q\in \partial D$, let us consider a (non-unique) conformal map $w\colon (D,q)\to (\mathbb H,\infty)$. 
For a real parameter $\mathfrak{b},$ $\Phi_{(\mathfrak{b})}$ stands for a background charge modification of $\Phi_{(0)}$ by
\[
\Phi_{(\mathfrak{b})}=\Phi_{(0)}-2\mathfrak{b}\arg w',
\]
see \cite[Section 10.1]{MK}, where we interchanged $\mathfrak{b}=b$ in order to be compatible with the previous text. The current field $J_{(\mathfrak{b})}=\partial \Phi_{(\mathfrak{b})}$ is  a pre-Schwarzian form of order $i\mathfrak{b}$, i.e., $J_{(\mathfrak{b})}=J_{(0)}+i\mathfrak{b}\,{w''}/{w'}$. The field $\Phi_{(\mathfrak{b})}$ has a stress tensor, and the Virasoro field is given by $T_{(\mathfrak{b})}=-\frac{1}{2}J*J+i\mathfrak{b}\,\partial J$. The central charge of $\Phi_{(\mathfrak{b})}$ is $c=1-12\mathfrak{b}^2$.

Let $a = \sqrt{2/\kappa}$ and $\mathfrak{b} = \sqrt{\kappa/8} - \sqrt{2/\kappa}.$ 
For $u = \alpha a\,\im\, w,$ we define a further modification $\Phi_{(\mathfrak{b},u)}$ of $\Phi_{(\mathfrak{b})}$ by
$$\Phi_{(\mathfrak{b},u)} = \Phi_{(\mathfrak{b})} +u, \quad J_{(\mathfrak{b},u)} =\partial \Phi_{(\mathfrak{b},u)}, \quad T_{(\mathfrak{b},u)} = T_{(\mathfrak{b})} - (\partial u) J_{(\mathfrak{b})} +i\mathfrak{b}\partial^2 u -\frac12(\partial u)^2$$
as in \cite[Proposition 10.5]{MK}.
This proposition also says that the vertex fields $\mathcal{V}_{(\mathfrak{b},u)}=e^{*a \Phi_{(\mathfrak{b},u)} }$ with $2a(a+\mathfrak{b})=1$ produce degenerate level two singular vectors. 
In a similar way, the OPE exponentials $V_{(\mathfrak{b},u)}^{ia}$ of $\Phi_{(\mathfrak{b},u)}^+(\cdot,q)$ produce degenerate level two singular vectors
\begin{equation} \label{eq: level2}
T*V^{ia} = \frac1{2a^2} \partial^2 V_{(\mathfrak{b},u)}^{ia},
\end{equation}
where $T = T_{(\mathfrak{b},u)}.$ 
Alternatively, this can be shown by using the action of the Virasoro generators $L_n,$ see \cite[Section 7.4]{MK}, and the current generators $J_n,$ see \cite[Section 11.1]{MK}. 
The field $V_{(\mathfrak{b},u)}^{ia}$ is a current primary holomorphic field of charge $q=a.$ 
The OPE exponential $V_{(\mathfrak{b},u)}^{ia}$ can be defined as a differential of conformal dimension $\lambda = (6-\kappa)/(2\kappa)$ at $z$, and
$$V_{(\mathfrak{b},u)}^{ia}(z) = e^{\frac12\alpha a^2 z} e^{\odot ia \Phi_{(\mathfrak{b},u)}^+(z,q)}$$ 
in the $\mathbb{H}$-uniformization.

\subsection{Connection to the chordal SLE with drift}
Let $p\in\partial D, p\ne q.$ Denote 
$$\widehat{\mathsf{E}}[\mathcal{X}] :=\frac{\mathsf{E}[V_{(\mathfrak{b},u)}^{ia}\mathcal{X}] }{\mathsf{E}[V_{(\mathfrak{b},u)}^{ia}] } = \mathsf{E}[e^{\odot ia \Phi^+_{(0)}(p,q)}\mathcal{X}],$$
where $\mathcal{X}$ is a Fock space functional/field. 
The insertion of $ e^{\odot ia \Phi^+_{(0)}(p,q)}= V_{(\mathfrak{b},u)}^{ia}/\mathsf{E}V_{(\mathfrak{b},u)}^{ia}$   
is an operator $\mathcal{X}\mapsto \widehat{\mathcal{X}}$ defined on the Fock space functionals/fields. 
It is a linear operator commuting with the differential operators $\partial, \bar\partial$ and preserving Wick's product.
For example, with $u = \alpha a\,\im\, w,$ 
$$\widehat{\Phi}_{(\mathfrak{b},u)} =\Phi_{(\mathfrak{b},u)} +2a\arg w = \Phi_{(\mathfrak{b})} +\alpha a \,\im\, w +2a\arg w.$$

The connection between the conformal field theory and the chordal SLE (\cite[Proposition 14.3]{MK}) can be extended to the SLE with drift.
Let $g_t$ be the one-parameter family of L\"owner maps driven by $\xi_t=\sqrt\kappa B_t + \alpha t:$
$$\partial_t g_t(z) = \frac2{g_t(z) - \xi_t},$$
where $B_t$ is a standard Brownian motion starting from $0$, and $g_0:(D,p,q)\to(\mathbb{H},0,\infty)$ is a given conformal map.  
Denote by $\mathrm{SLE}_\alpha(\kappa)$ the L\"owner evolution driven by $\xi_t=\sqrt\kappa B_t +\alpha t.$
Let $$D_t:=\{z\in D: \tau_z>t\},$$
where $\tau_z$ is the first time when the solution $g_t(z)$ to the L\"owner equation does not exist. The $\mathrm{SLE}_\alpha(\kappa)$ curve $\gamma$ is defined by $$\gamma_t=\lim_{z\to0}g_t^{-1}(z+\xi_t).$$ 
Then for each $t,$ $w_t := g_t - \xi_t$ is a well-defined random conformal map from $(D_t,\gamma_t,q)$ onto $(\mathbb{H},0,\infty).$

Denote by $\mathcal{F}_{(\mathfrak{b},u)}$ the OPE family of $\Phi_{(\mathfrak{b},u)}$ with $u = \alpha a \operatorname{Im} w.$
We call a non-random field $M$ a martingale-observable for $\mathrm{SLE}_\alpha(\kappa$) if the process
$$M_t(z_1,\cdots,z_n) = M_{D_t,\gamma_t,q}(z_1,\cdots,z_n)$$
is a local martingale on SLE probability space. Of course, it is stopped when for any of $z_j$'s  there exits $D_t.$

\begin{theorem} If $X_j\in\mathcal{F}_{(\mathfrak{b},u)}\, (u = \alpha a\,\im\, w),$ then the non-random fields
$$M(z_1,\cdots,z_n) =  \widehat{\mathsf{E}}[X_1(z_1)\cdots X_n(z_n)]$$
are martingale-observables for the chordal $\mathrm{SLE}_\alpha(\kappa)$ with drift.
\end{theorem}

\begin{proof}
Let $\xi\in \mathbb{R}$, and let $V \equiv V_{(\mathfrak{b},u)}^{ia}.$ 
We have  
$$ \mathsf{E}[V(\xi)]=e^{\frac12\alpha a^2\xi}$$ in the $\mathbb{H}$-uniformization,
It follows from Ward's equation that 
\begin{equation} \label{eq: Ward1}
\mathsf{E}[T*V(\xi)X] =  \mathsf{E}[T(\xi)]\, \mathsf{E}[V(\xi)X] +  \mathsf{E}[\mathcal{L}_{v_\xi}V(\xi)X],
\end{equation}
for $X = X_1(z_1)\cdots X_n(z_n),\,X\in\mathcal{F}_{(\mathfrak{b},u)},$
where $v_\xi(z) = 1/(\xi-z).$
In the $\mathbb{H}$-uniformization, we have 
\begin{equation}\label{eq: ET}
\mathsf{E}[T(\xi)] = -\frac12(\partial u)^2 = \frac{\alpha^2a^2}8.
\end{equation}
Let 
$$R_\xi\equiv\widehat{ \mathsf{E}}_{\xi}[X] := \frac{ \mathsf{E}[V(\xi)X]}{ \mathsf{E}[V(\xi)]}=\mathsf{E}[e^{\odot ia \Phi^+_{(0)}(\xi,q)}\mathcal{X}] = e^{-\frac12\alpha a^2\xi}\,\mathsf{E}[V(\xi)X].$$ 
Then we obtain
$$\partial_\xi R_\xi = -\frac12\alpha a^2 R_\xi+ e^{-\frac12\alpha a^2\xi} \, \mathsf{E}[\partial_\xi V(\xi)X], $$
and 
\begin{align*}
\partial_\xi^2 R_\xi = \Big(\frac12\alpha a^2\Big)^2 R_\xi - \alpha a^2 e^{-\frac12\alpha a^2\xi}  \,\mathsf{E}[\partial_\xi V(\xi)X] + e^{-\frac12\alpha a^2\xi} \, \mathsf{E}[\partial_\xi^2 V(\xi)X].
\end{align*}
By the level two degeneracy equation \eqref{eq: level2} for $V,$ we arrive at
$$\frac1{2a^2}\partial_\xi^2 R_\xi =  \mathsf{E}[T(\xi)]\, R_\xi - \frac12 \alpha (\partial_\xi R_\xi+\frac12\alpha a^2 R_\xi)+ e^{-\frac12\alpha a^2\xi}\,  \mathsf{E}[T* V(\xi)X].$$
It follows from \eqref{eq: Ward1} and \eqref{eq: ET}, that 
$$\frac1{2a^2}\partial_\xi^2 \,\widehat{ \mathsf{E}}_{\xi}[X] = -\frac12 \alpha \partial_\xi\, \widehat{ \mathsf{E}}_{\xi}[X] +  \widehat{ \mathsf{E}}_{\xi}[\mathcal{L}_{v_\xi} X].$$
This is the BPZ-Cardy equation for conformal field theory of the chordal $\mathrm{SLE}_\alpha(\kappa).$ 

Now the process $M_t := M_{D_t,\gamma_t,q}$ is represented in terms of $R_\xi,$ $\xi_t,$ and $g_t:$  
$$M_t = m(\xi_t,t), \qquad m(\xi,t) =\big(R_\xi\,\|\,g_t^{-1}\big).$$
By It\^o's formula we have
$$dM_t=\partial_\xi\Big|_{\xi=\xi_t}~m(\xi,t)\,(\sqrt\kappa dB_t +\alpha\,dt)+\frac\kappa2\partial^2_\xi\Big|_{\xi=\xi_t}~m(\xi,t)\,dt+L_t\, dt,$$
where 
$$L_t :=\frac{d}{ds}\Big|_{s=0}(R_{\xi_t}\,\|\,g_{t+s}^{-1}).$$
As in the chordal SLE without drift, we have 
$L_t=-2\big(\mathcal{L}_{v_{\xi_t}}R_{\xi_t}\,\|\,g_{t}^{-1}\big).$
By the BPZ-Cardy equation, we have
$$L_t=  - \frac1{a^2}(\partial_{\xi}^2\big|_{\xi=\xi_t}R_\xi\,\|\,g_t^{-1})-\alpha(\partial_{\xi}\big|_{\xi=\xi_t}R_\xi\,\|\,g_t^{-1}),$$
and thus, $dM_t$ has a vanishing drift term.
\end{proof}

\begin{remark*}
A physical interpretation of the boundary condition changing operator will be given later,
see the last remark in Subsection~\ref{ss: dipolar}. 
\end{remark*}

\section{CFT of dipolar SLE with drift}
This part of the paper was inspired by analogous reasonings in \cite{Kang2012a}.
Let $D$ be a simply connected proper subdomain of $\mathbb{C}$, and let $q_-,q_+\in \partial D$ be two marked boundary points. We write $Q$ for the marked boundary arc from $q_+$ to $q_-$ in the positive direction.  

A dipolar $\mathrm{SLE}_\alpha(\kappa)$ driven by $\widetilde\xi_t = \sqrt\kappa B_t + \alpha t$ is defined by 
$$\partial_t \widetilde{g}_t(z) = \coth\big(\frac{\widetilde{g}_t(z)-\widetilde\xi_t}2\big), \qquad B_0=0,$$
where $\widetilde{g}_0$ is a given conformal map from $D$ onto the strip $\mathbb{S} := \{z\in\mathbb{C}: 0<\im z < \pi\}$ which maps $Q$ onto the upper boundary $\pi i + \mathbb{R}$ of $\mathbb{S}.$
Let 
$$\widetilde w_t = \widetilde g_t - \widetilde\xi_t,\quad \varphi(z) :=\tanh(z/2),\quad \xi_t=\varphi(\widetilde\xi_t), \quad g_t = \varphi\circ\widetilde g_t, \quad w_t = \varphi\circ\widetilde w_t.$$
Then $g_t$ satisfies
$$\partial_t g_t (z)=-\frac{1-g_t^2(z)}{2}\frac{1-\xi_t g_t(z)}{\xi_t-g_t(z)}.$$
Also, it follows from It\^o's formula that 
$$dw_t = -\frac\alpha2(1-w_t^2)\,dt + \frac{1-w_t^2}{2w_t}\,dt -\frac\kappa4w_t(1-w_t^2)\,dt-\frac{\sqrt\kappa}2(1-w_t^2)\,dB_t.$$
Note that $w_t$ is a normalized slit holomorphic stochastic flow from $D_t:=\{z\in D: \tau_z>t\}$ (as in the chordal case, $\tau_z$ is the first time when the solution $g_t(z)$ to the L\"owner equation does not exist) onto the upper half-plane $\mathbb{H}$ with the semicomplete vector field 
$$b(z) = \frac\alpha2(1-z^2) - \frac{1-z^2}{2z}$$ 
and the complete vector field 
$$\sigma(z) = -\frac{1-z^2}2.$$
The fields $b$ and $\sigma$ are equivalent to those in Proposition~\ref{th:classification} ~\eqref{th:classification dipolar} via $\phi:\mathbb{H}\to\mathbb{H},z\mapsto2z.$
Furthermore, the process $\xi_t$ satisfies
$$d\xi_t =  \frac\alpha2(1-\xi_t^2)\,dt -\frac\kappa4\xi_t(1-\xi_t^2)\,dt + \frac{\sqrt\kappa}2(1-\xi_t^2)\,dB_t.$$

\subsection{Review of dipolar CFT}
We consider a conformal transformation $w$ from $(D,q_-,q_+)$ onto $(\mathbb{H},-1,1).$
Recall  that we set $\mathfrak{b} = \sqrt{\kappa/8} - \sqrt{2/\kappa}.$ Let us define a background charge modifications $\Phi_{(\mathfrak{b})}$ by 
$$\Phi_{(\mathfrak{b})} = \Phi_{(0)} -2\mathfrak{b}\arg\big(\frac{w'}{1-w^2}\big).$$
Note that $\Phi_{(\mathfrak{b})}$ does not depend on the choice of $w.$
As in the chordal case, the current field $J_{(\mathfrak{b})}$ and the Virasoro field $T_{(\mathfrak{b})}$ change accordingly.

Before we introduce the boundary condition changing operator, let us recall the definition of OPE exponentials $\mathcal{O}^{(\tau,\tau_*;\tau_-,\tau_+)}:$ 
$$\mathcal{O}^{(\tau,\tau_*;\tau_-,\tau_+)}(z) = M^{(\tau,\tau_*;\tau_-,\tau_+)}(z) \, e^{\odot i(\tau\Phi^+_{(0)}(z)-\tau_*\Phi^-_{(0)}(z)+\tau_-\Phi^+_{(0)}(q_-)+\tau_+\Phi^+_{(0)}(q_+))}$$
under the neutrality condition $(\mathrm{NC}_0):$ 
$$\tau + \tau_*+\tau_-+\tau_+=0.$$ 
Here, the correlation part $M^{(\tau,\tau_*;\tau_-,\tau_+)} = \mathsf{E}\,[\mathcal{O}^{(\tau,\tau_*;\tau_-,\tau_+)}]$ is given by
\begin{align}\label{eq: M4O}
M^{(\tau,\tau_*;\tau_-,\tau_+)}&=(1-w)^{\nu^+}(1+w)^{\nu^-} (1-\bar{w})^{\nu_*^+}(1+\bar{w})^{\nu_*^-}(w-\bar{w})^{\tau \tau_*}
\\&\times
({w'_-})^{\lambda_-}({w'_+})^{\lambda_+}(w')^{\lambda}(\overline{w'})^{\lambda_*}, \qquad (w'_\pm=w'(q_\pm)). \nonumber
\end{align}
The dimensions $[\lambda,\lambda_*;\lambda_-,\lambda_+]$ and exponents $\nu^\pm,\nu_*^\pm$ are 
$$\lambda=\frac{\tau^2}{2}-\tau\mathfrak{b}, \quad \lambda_*=\frac{\tau_*^2}{2}-\tau_*\mathfrak{b} , \quad \lambda_\pm=\frac{\tau_\pm^2}{2},$$
and $\nu^\pm=\tau(\mathfrak{b}+\tau_\pm),\,\,
\nu_*^\pm=\tau_*(\mathfrak{b}+\tau_\pm).$ 

On the other hand, the Wick part of $\mathcal{O}^{(\tau,\tau_*;\tau_-,\tau_+)}(z)$ is  Wick's exponential of the formal fields
$$i(\tau\Phi^+_{(0)}(z)-\tau_*\Phi^-_{(0)}(z)+\tau_-\Phi^+_{(0)}(q_-)+\tau_+\Phi^+_{(0)}(q_+)),$$
where the formal fields $\Phi^\pm_{(0)}$ have the formal correlations
$$\mathsf{E}[\Phi^+_{(0)}(z)\Phi^+_{(0)}(z_0)] = \log\frac1{w(z)-w(z_0)},\quad \mathsf{E}[\Phi^+_{(0)}(z)\overline{\Phi^-_{(0)}(z_0)}] = \log(w(z)-\overline{w(z_0)}),$$
and satisfy the relations 
$$\Phi^-_{(0)}=\overline{\Phi^+_{(0)}} \textrm{ and } \Phi_{(0)} = \Phi^+_{(0)}+\Phi^-_{(0)} = 2\re\, \Phi^+_{(0)}.$$
Note that Wick's exponential is a well-defined Fock space field as long as the neutrality condition $(\mathrm{NC}_0)$ holds.
Furthermore, the OPE exponentials $\mathcal{O} \equiv \mathcal{O}^{(\tau,\tau_*;\tau_-,\tau_+)}$ satisfy Ward's OPE
$$T(\zeta)\mathcal{O}(z) \sim \lambda\frac{\mathcal{O}(z)}{(\zeta -z)^2}+\frac{\partial_z\mathcal{O}(z)}{\zeta-z}, \qquad(\zeta \to z),$$
and a similar OPE holds for $\bar{\mathcal{O}}$ with $\bar\lambda_*.$ 
See \cite[Proposition~4.2]{Kang2013}.

\subsection{Connection to the dipolar SLE with drift} \label{ss: dipolar}

Fix $a=\sqrt{2/\kappa}$ and $\mathfrak{b} =  \sqrt{\kappa/8} - \sqrt{2/\kappa}$ as usual. 
For a real parameter $\delta,$  we now introduce the boundary condition changing insertion field  
$$\Psi^\delta(\xi):=\mathcal{O}^{(a,0;  -(a-\delta)/2, - (a+\delta)/2)}(\xi).$$
The insertion of 
$$\frac{\Psi^\delta(\xi)}{\mathsf{E}[\Psi^\delta(\xi)]} = \exp\Big(\odot i \big(a\Phi^+_{(0)}(\xi)-\frac{a-\delta}2\Phi_{(0)}^+(q_-) - \frac{a+\delta}2\Phi_{(0)}^+(q_+)\big)\Big)$$
is an operator $\mathcal{X}\to\widehat{\mathcal{X}}$ on Fock space functionals/fields. 
For example,
$$\widehat\Phi_{(\mathfrak{b})} = \Phi_{(\mathfrak{b})} +2a\,\arg w - (a-\delta)\, \arg(1+w) - (a+\delta)\, \arg(1-w),$$
where $w$ is a conformal map from $(D,p,q_-,q_+)$ onto $(\mathbb{H},0,-1,1).$
The field $\Psi^\delta$ is a holomorphic current primary field of charge $a$ and it has conformal dimension $h = (6-\kappa)/(2\kappa).$
It satisfies the level two degeneracy equation
\begin{equation} \label{eq: level2d}
T_{(\mathfrak{b})}*\Psi^\delta = \frac\kappa4\, \partial^2\Psi^\delta.
\end{equation}
Let 
$$\widehat{ \mathsf{E}}[X] := \frac{ \mathsf{E}[\Psi^\delta(p)X]}{ \mathsf{E}[\Psi^\delta(p)]}.$$ 
\begin{theorem}
If $X_j\in\mathcal{F}_{(\mathfrak{b})}$ and if $\delta = \alpha a,$ then the non-random fields
$$M(z_1,\cdots,z_n) =  \widehat{\mathsf{E}}[X_1(z_1)\cdots X_n(z_n)]$$
are martingale-observables for the dipolar $\mathrm{SLE}_\alpha(\kappa)$ with drift.
\end{theorem}

\begin{proof}
Combining the level two degeneracy equation~\eqref{eq: level2d} for $\Psi^\delta$ with Ward's equation \cite[Proposition~3.3]{Kang2013} we have 
\begin{align} \label{eq: Cardy0}
 \mathsf{E}[\Psi^\delta(\xi)\mathcal{L}_{v_\xi}X] &= \frac1{2a^2}\frac{(1-\xi^2)^2}2\, \mathsf{E}[\partial_\xi^2\Psi^\delta(\xi)X] -\frac{3\xi(1-\xi^2)}2\, \mathsf{E}[\partial_\xi\Psi^\delta(\xi)X]\\
&+\Big(\frac{3\xi^2-1}2h+\mathfrak{b}^2-h_+-h_-\Big)\, \mathsf{E}[\Psi^\delta(\xi)X] \nonumber
\end{align}
for the tensor product $X = X_(z_1)\cdots X_n(z_n)$ of fields $X_j$ in $\mathcal{F}_{(\mathfrak{b})}.$
Here, 
$$h  = \frac{6-\kappa}{2\kappa}, \quad h_\pm = \frac12\tau_\pm^2 = \frac18(a\pm\delta)^2$$
are dimensions of $\Psi^\delta$ at $\xi,q_\pm$, and $v_\xi$ is the usual vector field, 
$$v_\xi(z) = \frac{1-z^2}2\frac{1-\xi z}{\xi-z}.$$
Note that all fields  in \eqref{eq: Cardy0} are evaluated in the identity chart of upper half-plane.

Let 
$$R_\xi\equiv\widehat{ \mathsf{E}}_\xi[X] := \frac{ \mathsf{E}[\Psi^\delta(\xi)X]}{ \mathsf{E}[\Psi^\delta(\xi)]}.$$ 
Recall that (see \eqref{eq: M4O}) 
$$ \mathsf{E}[\Psi^\delta(\xi)] = M^{(a,0;-(a-\delta)/2,-(a+\delta)/2)}(\xi) = (1-\xi)^{\nu^+}(1+\xi)^{\nu^-}$$
in the upper half-plane,
where $\nu_\pm = -h \mp a\delta/2.$
Due to the relations 
$$\frac{\partial_\xi\,\mathsf{E}[\Psi^\delta(\xi)X]}{\mathsf{E}[\Psi^\delta(\xi)X]} =\frac{\partial_\xi R_\xi}{R_\xi} + \frac{\nu_+}{\xi-1} +  \frac{\nu_-}{\xi+1}$$
and 
$$\frac{\partial_\xi^2\,\mathsf{E}[\Psi^\delta(\xi)X]}{\mathsf{E}[\Psi^\delta(\xi)X]}
-\left(\frac{\partial_\xi\,\mathsf{E}[\Psi^\delta(\xi)X]}{\mathsf{E}[\Psi^\delta(\xi)X]}\right)^2
 =\frac{\partial_\xi^2 R_\xi}{R_\xi} - \left(\frac{\partial_\xi R_\xi}{R_\xi}\right)^2 - \frac{\nu_+}{(\xi-1)^2} -  \frac{\nu_-}{(\xi+1)^2},$$
we rewrite \eqref{eq: Cardy0} as 
\begin{equation}\label{eq: Cardy1}
\widehat{ \mathsf{E}}_\xi[\mathcal{L}_{v_\xi}X] = \frac\delta{2a}(1-\xi^2)\partial R_\xi + \frac1{2a^2}\Big(\frac{(1-\xi^2)^2}2\partial^2R_\xi - \xi(1-\xi^2)\partial R_\xi\Big).
\end{equation}
This is the BPZ-Cardy equation (cf. \cite[Proposition 5.1]{Kang2013}) for the conformal field theory of the dipolar SLE with drift.

Now the process $M_t (z_1,\cdots,z_n)\equiv M_{(D_t ,\gamma_t , Q)}(z_1,\cdots, z_n)$
is represented as 
$$M_t =m(\xi_t ,t) ,\,\, m(\xi, t )=(R_{\xi} \,\|\, g^{-1}_t ).$$
By It\^o's formula, $M_t$ is a semi-martingale with the drift term
$$\Big(\frac\alpha2(1-\xi^2)\partial_\xi +\frac{\kappa}{4}\big(\frac{(1-\xi^2)^2}{2}\partial_{\xi}^2 -\xi (1-\xi^2)\partial_{\xi}\big)\Big)\Big|_{\xi=\xi_t}m(\xi ,t)\,\,d t +\frac{d}{ds}\Big|_{s=0}(R_{\xi}\,\|\, g_{t+s}^{-1})\,\,d t.$$
Similarly to the chordal case, we have 
$$\frac{d}{ds} \Big|_{s=0}(R_{\xi}\,\|\, g_{t+s}^{-1})=-(\mathcal{L}_{v_{\xi_t}}R_{\xi_t}\,\|\,g_t^{-1}).$$
It follows from the BPZ-Cardy equation that 
$$\frac{d}{ds} \Big|_{s=0}(R_{\xi}\,\|\, g_{t+s}^{-1}) = -\Big(\frac\delta{2a}(1-\xi^2)\partial_\xi +\frac{\kappa}{4}\big(\frac{(1-\xi^2)^2}{2}\partial_{\xi}^2 -\xi (1-\xi^2)\partial_{\xi}\big)\Big)\Big|_{\xi=\xi_t}m(\xi ,t)\,\,d t .$$
Hence, the drift term of $M_t$ vanishes if we choose
$\delta = \alpha a.$
\end{proof}

\begin{remark*}
For a positive real $q$ and $\alpha\in\mathbb{R},$ let  $\delta = (\alpha a /2)\cdot q$ and  
$$ \Psi^\delta(0,-q,q) = \mathcal{O}^{(a, 0;-(a-\delta)/2, - (a+\delta)/2)}(0,-q,q)$$
in the upper half-plane. 
Denote
$$\widehat{\mathsf{E}}^q[\mathcal{X}]:=\frac{\mathsf{E}[\Psi^\delta(0,-q,q)\mathcal{X}]}{\mathsf{E}[\Psi^\delta(0,-q,q)]}.$$
For example, 
$$\widehat{\mathsf{E}}^q[\,\Phi_{(0)}(z)] = 2a\arg z -(a+\delta)\arg(q-z)  -(a-\delta)\arg(q+z)$$
in the upper half-plane. 
As $q\to\infty,$ we have 
$$\widehat{\mathsf{E}}^q[\,\Phi_{(0)}(z)]  \to 2a\arg z+\alpha a \im\,z = \widehat{\mathsf{E}}[\,\Phi_{(0,u)}(z)],$$
where $u(z) = \alpha a \im\,z.$
Thus one can view the conformal field theory of the chordal $\mathrm{SLE}_\alpha(\kappa)$ as the limiting theory of that of dipolar $\mathrm{SLE}_\alpha(\kappa)$ as two boundary marked points merge. 
\end{remark*}

\subsection{Cardy-Zhan's formula}
We now extend Cardy-Zhan's formula (see \cite{zhan}) in the dipolar SLE without drift to the case $\alpha\in (-1,1).$ 

Denote by $K_t:=\{z\in\mathbb{S}: \tau_z \le t\}$ the hull of dipolar SLE with drift.
Then $\bar{\mathbb{S}}\setminus\bar K_\infty$ has two component, $S_\pm\,(-\infty\in S_-,+\infty\in S_+).$
Now we suppose $\kappa >4$ and $ABC$ is a triangle with 
$$\angle A = (1-\frac4\kappa)\pi, \quad \angle B =  \frac2\kappa(1+\alpha)\pi,  \quad \angle C =  \frac2\kappa(1-\alpha)\pi.$$
Consider a Schwarz-Christoffel map $f$ from a strip $\mathbb{S}$ onto the triangle $ABC$ such that $$f(0) = A,  \quad f(+\infty) = B,  \quad  f(-\infty) = C.$$

\begin{proposition} \label{Cardy-Zhan}
Suppose $f(z) = aA + bB+ cC$ with $a+b+c = 1, a,b,c\in\mathbb{R}.$
Then 
$$\mathbb{P}[z\in K_\infty]=a,\qquad \mathbb{P}[z\in S_+] = b,\qquad \mathbb{P}[z\in S_-]=c.$$
\end{proposition}

\begin{proof}
Let $Z_t = \widetilde g_t - \widetilde\xi_t, \widetilde\xi_t =\sqrt\kappa B_t + \alpha t.$
Then by It\^o's formula, we have 
$$df(Z_t) = -\sqrt\kappa f'(Z_t)\, dB_t + \Big(\frac\kappa2 f''(Z_t) + \big(\coth\frac{Z_t}2-\alpha\big)f'(Z_t)\Big)dt.$$
Let $h(z) = f(\log z).$
The process $f(Z_t(z))$ is a local martingale if and only if the map $h$ satisfies 
\begin{equation}\label{eq: SC4h}
\frac{h''(z)}{h'(z)} = -\frac4\kappa\frac1{z-1} + \Big(-1+\frac2\kappa(1+\alpha)\Big)\frac1z.
\end{equation}
Due to our choice of Schwarz-Christoffel map $f,$ the function $h$ maps $\mathbb{H}$ conformally onto the triangle $ABC$ and $h(1) = A, \, h(\infty) = B$ and $h(0) = C.$
Thus $h$ satisfies \eqref{eq: SC4h}.
Note that $f$ is bounded and its continuous extension to $\partial \mathbb{S}$ is analytic on each component of $\partial \mathbb{S}\setminus\{0\}.$
Thus for $z\in \bar{\mathbb{S}}\setminus\{0\},$ the stopped process $f(Z_t(z)),\,(0\le t < \tau_z)$ is a martingale.
If $z\in K_\infty, S_+,S_-,$ respectively, then $f(Z_t(z))$ tends to $A, B, C,$ respectively, as $t\to\tau_z.$
Proposition now follows from the optional stopping time theorem. 
\end{proof}

In the last remark below, we identify the Cardy-Zhan observables with martingale-observables derived from conformal field theory of dipoalr SLE with drift. 
First, note that under the insertion of ${\Psi^\delta(p)}/{\mathsf{E}[\Psi^\delta(p)]},$ the OPE exponentials $\mathcal{O}^{(\tau,\tau_*;\tau_-,\tau_+)}$ with the neutrality condition $(\mathrm{NC}_0)$ have the correlation functions (1-point vertex-observables) $\widehat{M}^{(\tau,\tau_*;\tau_-,\tau_+)} = \widehat{\mathsf{E}}^\delta[\mathcal{O}^{(\tau,\tau_*;\tau_-,\tau_+)}]:$ 
\begin{align}\label{eq: M4Ohat}
\widehat M^{(\tau,\tau_*;\tau_-,\tau_+)}&=(1-w)^{\widehat\nu^+}(1+w)^{\widehat\nu^-} (1-\bar{w})^{\widehat\nu_*^+}(1+\bar{w})^{\widehat\nu_*^-}(w-\bar{w})^{\tau \tau_*} w^{\tau a}{\bar w}^{\tau_* a}
\\&\times
({w'_-})^{\widehat\lambda_-}({w'_+})^{\widehat\lambda_+}(w')^{\lambda}(\overline{w'})^{\lambda_*}, \qquad (w'_\pm=w'(q_\pm)), \nonumber
\end{align}
where the dimensions $[\lambda,\lambda_*;\widehat\lambda_-,\widehat\lambda_+]$ and exponents $\widehat\nu^\pm,\widehat\nu_*^\pm$ are 
$$\lambda=\frac{\tau^2}{2}-\tau\mathfrak{b}, \quad \lambda_*=\frac{\tau_*^2}{2}-\tau_*\mathfrak{b} , \quad \widehat\lambda_\pm=\frac{\tau_\pm^2-(a\pm\delta)\tau_\pm}{2},$$
and $\widehat\nu^\pm=\tau(\mathfrak{b}-(a\pm\delta)/2+\tau_\pm),\,\,
\widehat\nu_*^\pm=\tau_*(\mathfrak{b}-(a\pm\delta)/2+\tau_\pm).$ 
The changes from \eqref{eq: M4O} to \eqref{eq: M4Ohat}  can be derived from the rooting procedure (see \cite{MK,Kang2013}) or algebra of OPE exponentials (see \cite{Kang2012a}).

\begin{remark*}
The Cardy-Zhan observables in Proposition~\ref{Cardy-Zhan} are real-valued non-random fields with all conformal dimensions zero.
The constant field is the only vertex observable satisfying such dimension requirements. 
However, one can identify their derivative with a vertex-observable which has conformal dimensions
$$\lambda=1, \quad \lambda_{*}=\widehat\lambda_{+}=\widehat\lambda_{-}=0.$$
Note that a vertex observable $\widehat{M}^{(-2a,0;a-\delta,a+\delta)}$ has such dimensions.
In the $(\mathbb{H},0,\pm1)$-uniformization, we have 
$$\widehat{M}^{(-2a,0;a-\delta,a+\delta)} = (1-z)^{-1+\frac2\kappa(1-\alpha)}(1+z)^{-1+\frac2\kappa(1+\alpha)}z^{-\frac4\kappa}.$$
It follows that 
$$\frac{\partial  \widehat{M}^{(-2a,0;a-\delta,a+\delta)}}{ \widehat{M}^{(-2a,0;a-\delta,a+\delta)}} = -\frac4\kappa\frac1{z} + \Big(-1+\frac2\kappa(1-\alpha)\Big)\frac1{z-1}+\Big(-1+\frac2\kappa(1+\alpha)\Big)\frac1{z+1}.$$
Compare this in the $(\mathbb{H},0,-1,1)$-uniformization to \eqref{eq: SC4h} in the $(\mathbb{H},1,0,\infty)$-uniformization.
\end{remark*}

\end{document}